\documentclass[a4paper,11pt]{article}
\usepackage{graphicx}
\usepackage{amsfonts}
\usepackage{amsmath}
\usepackage{amssymb}
\usepackage{amsthm}
\usepackage{hyperref}
\usepackage{xcolor}
\title{New orientable sequences}
\author{Chris J. Mitchell and Peter R. Wild
\\Information Security Group, Royal Holloway, University of London\\
\href{mailto:me@chrismitchell.net}{me@chrismitchell.net};
~~~~\href{mailto:peterrwild@gmail.com}{peterrwild@gmail.com}}
\date{3rd July 2025}
\newtheorem{lemma}{Lemma}[section]
\newtheorem{corollary}[lemma]{Corollary}

\newtheorem{theorem}[lemma]{Theorem}
\newtheorem{definition}[lemma]{Definition}
\newtheorem{remark}[lemma]{Remark}
\newtheorem{example}[lemma]{Example}

\begin{document}

\maketitle

\begin{abstract}
Orientable sequences of order $n$ are infinite periodic sequences with symbols drawn from a finite
alphabet of size $k$ with the property that any particular subsequence of length $n$ occurs at most
once in a period \emph{in either direction}.  They were introduced in the early 1990s in the
context of possible applications in position sensing. Bounds on the period of such sequences and a
range of methods of construction have been devised, although apart from very small cases a
significant gap remains between the largest known period for such a sequence and the best known
upper bound. In this paper we first give improved upper bounds on the period of such sequences. We
then give a new general method of construction for orientable sequences involving subgraphs of the
de Bruijn graph with special properties, and describe two different approaches for generating such
subgraphs. This enables us to construct orientable sequences with periods meeting the improved
upper bounds when $n$ is 2 or 3, as well as $n=4$ and $k$ odd. For $4\leq n\leq 8$, in some cases
the sequences produced by the methods described have periods larger than for any previously known
sequences.
\end{abstract}

\section{Introduction} \label{section:Introduction}

\subsection{Orientable sequences}

Orientable sequences were introduced in the early 1990s \cite{Burns92,Burns93,Dai93} in the context
of possible applications in position sensing.  An orientable sequence of order $n$ is an infinite
periodic sequence with symbols drawn from a finite alphabet --- typically $\mathbb{Z}_k$ for some
$k$ --- with the property that any particular subsequence of length $n$, referred to throughout as
an $n$-tuple, occurs at most once in a period \emph{in either direction}.  That is, if anyone
reading the sequence observes $n$ consecutive symbols, they can deduce both the direction in which
they are reading and their position within one period of the sequence. Gabri\'{c} and Sawada
\cite{Gabric25} provide an interesting discussion of further possible applications as well as their
relationship to strings relevant to DNA computing.

The early work referred to above focussed on the binary case, i.e.\ where $k=2$; Dai et al.\
\cite{Dai93} provided both an upper bound on the period for this case and a method of construction
yielding sequences with asymptotically optimal periods. More recently, Mitchell and Wild
\cite{Mitchell22} showed how the Lempel homomorphism \cite{Lempel70} could be applied to
recursively generate binary orientable sequences with periods a large fraction of the optimal
value.  In 2024, Gabri\'{c} and Sawada \cite{Gabric24,Gabric24b} described a highly efficient
method of generating binary orientable sequences with the largest known periods.

In 2024, Alhakim et al.\ \cite{Alhakim24a} studied the general alphabet case, i.e.\ where $k>2$.
They gave an upper bound on the period of orientable sequences for all $n$ and $k$, and also
described a range of methods of construction using the Alhakim and Akinwande generalisation of the
Lempel homomorphism to arbitrary finite alphabets \cite{Alhakim11}.  Since then a range of
construction methods have been proposed \cite{Gabric25,Mitchell24a,Mitchell25a}; of particular
interest is the method of Gabri\'{c} and Sawada \cite{Gabric25}, who showed how to construct
sequences with asymptotically optimal period for any $n$ and $k>2$ using a cycle-joining approach.

In this paper we first give new upper bounds on the period of an orientable sequence, which improve
on the previous bounds for $n\geq3$. We also describe a novel general approach to the construction
of orientable sequences for any $k$ and $n$. This approach involves showing, using graph-theoretic
arguments, that a subset of $n$-tuples with special properties can always be `joined' to create an
orientable sequence.  Two methods for generating sets of $n$-tuples with the appropriate set of
properties are described. The orientable sequences obtained have optimal periods for $n=2$, $n=3$
and $n=4$ when $k$ is odd. For $4\leq n\leq 8$, for some values of $k$ the sequences generated have
longer period than any previously known sequences, although sequences generated by the method of
Gabri\'{c} and Sawada \cite{Gabric25} have the largest known period for larger values of $n$.

\subsection{Basic definitions}

 In this paper we consider periodic sequences $(s_i)$ with elements from $\mathbb{Z}_k$
for some $k$, which we refer to as $k$-ary.  For a sequence $S = (s_i)$ and $n\geq1$
we write $\mathbf{s}_n(i) = (s_i,s_{i+1},\ldots,s_{i+n-1})$, i.e.\ a string of $n$ consecutive
symbols occurring in the sequence at positions $i,i+1,\ldots,i+n-1$.  We refer to such
strings as $n$-tuples. Since we are interested in tuples occurring either forwards or
backwards in a sequence, we also introduce the notion of a reversed tuple, so that if $\mathbf{u}
=
(u_0,u_1,\ldots,u_{n-1})$ is a $k$-ary $n$-tuple then $\mathbf{u}^R = (u_{n-1},u_{n-2},
\ldots,u_0)$ is its \emph{reverse}.  We are also interested in negating all the elements of a
tuple, and hence if $\mathbf{u} = (u_0,u_1,\ldots,u_{n-1})$ is a $k$-ary $n$-tuple, we write
$-\mathbf{u}$ for $(-u_0,-u_1,\ldots,-u_{n-1})$.

\begin{definition}[\cite{Alhakim24a}]
A $k$-ary \emph{$n$-window sequence $S = (s_i)$} is a periodic sequence of elements from
$\mathbb{Z}_k$ ($k>1$, $n>1$) with the property that no $n$-tuple appears more than once in a
period of the sequence, i.e.\ with the property that if $\mathbf{s}_n(i) = \mathbf{s}_n(j)$ for
some $i,j$, then $i \equiv j \pmod m$ where $m$ is the period of the sequence.
\end{definition}

A $k$-ary de Bruijn sequence \emph{of order $n$} is then simply an $n$-window sequence in which
every $k$-ary $n$-tuple appears once in a period, i.e.\ an $n$-window sequence of maximal period.

Following Alhakim et al.\ \cite{Alhakim24a} we also introduce the de Bruijn digraph. For positive
integers $n$ and $k$ greater than one, let $\mathbb{Z}_k^n$ be the set of all $k^n$ vectors of
length $n$ with entries from the group $\mathbb{Z}_k$ of residues modulo $k$. The order $n$ de
Bruijn digraph, $B_k(n)$, is a directed graph with $\mathbb{Z}^n_k$ as its vertex set in which,
for
any two vectors $\textbf{x} = (x_0,x_1,\ldots,x_{n-1})$ and $\textbf{y} =
(y_0,y_1,\ldots,y_{n-1})$, the pair $(\textbf{x},\textbf{y})$ is an edge if and only if $y_i =
x_{i+1}$ for every $i$ ($0\leq i< n-1$). We label such an edge with the $(n+1)$-tuple
$(x_0,x_1,\ldots,x_{n-1},y_{n-1})$.

Note that we have defined two ways of specifying an edge in $B_k(n)$, namely as either a pair of
vertices $(\textbf{a},\textbf{b})$, where $\textbf{a},\textbf{b}$ are $k$-ary $n$-tuples, or as a
single $k$-ary $(n+1)$-tuple $\textbf{x}$.  Note that, in this case, the $(n+1)$-tuple
$\textbf{x}^R$ denotes the same edge as $(\textbf{b}^R,\textbf{a}^R)$.

It is straightforward to verify that there is a correspondence between a $k$-ary de Bruijn
sequence
of order $n$ and a directed Eulerian circuit in $B_{k}(n-1)$, in which consecutive edges in the
circuit correspond to consecutive $n$-tuples in the sequence.

\begin{definition}[\cite{Alhakim24a}] A $k$-ary $n$-window sequence $S = (s_i)$ is
said to be an \emph{orientable sequence of order $n$}, an $\mathcal{OS}_k(n)$, if $\mathbf{s}_n(i)
\neq \mathbf{s}_n(j)^R$, for any $i,j$.
\end{definition}

\subsection{This paper}

The remainder of this paper is structured as follows.  We start, in
Section~\ref{section:upper_bound}, by developing new upper bounds on the period of an orientable
sequence. In Section~\ref{section:Eulerian_cycles} we describe our novel general method of
construction, which involves regarding $n$-tuples as edges in the de Bruijn graph of order $n-1$. A
first method of constructing a set of $n$-tuples with the desired properties is presented in
Section~\ref{section:construction}.  A second method, involving use of the Lempel homomorphism from
the de Bruijn graph of order $n$ to the de Bruijn graph of order $n-1$, is given in
Section~\ref{section:Lempel_construction}. For both constructions simple examples are given, and
the periods of the orientable sequences generated are tabulated for small values of $n$ and $k$.
This leads naturally to Section~\ref{section:other_work}, where a table is given of the largest
known periods for orientable sequences for $n$ and $k$ at most 8.  As discussed briefly in the
concluding section, Section~\ref{section:conclusions}, this table reveals that for $n>3$ there
remains a significant gap between the largest known period for an orientable sequence and the
existing upper bound.

Preliminary versions of some of the results in this paper were presented at Sequences 2025 (London,
February 2025) and the 5th Pythagorean Conference (Kalamata, June 2025).

\section{New upper bounds on the period of an orientable sequence}  \label{section:upper_bound}

\subsection{Preliminaries}  \label{subsection:preliminaries_for_bounds}

We start by giving some simple definitions and elementary results relating to these definitions.

\begin{definition}
Suppose $n\geq1$ and $k\geq2$.  If $\mathbf{a}=(a_0,a_1,\ldots,a_{n-1})$ is a $k$-ary $n$-tuple,
then $\mathbf{a}$ is said to be \emph{uniform} if and only if $a_i=a_j$ for every $i,j\in
\{0,1,\ldots,n-1\}$.
\end{definition}

\begin{definition}
Suppose $n\geq1$ and $k\geq2$.  If $\mathbf{a}=(a_0,a_1,\ldots,a_{n-1})$ is a $k$-ary $n$-tuple,
then $\mathbf{a}$ is said to be \emph{alternating} if and only if there exist $c_0$ and $c_1$ such
that $a_{2i}=c_0$ and $a_{2i+1}=c_1$ for every $i$ such that $i\ge0$ and $2i+1\le n-1$ and
$\mathbf{a}$ is not uniform.
\end{definition}

\begin{definition}
Suppose $n\geq1$ and $k\geq2$.  If $\mathbf{a}=(a_0,a_1,\ldots,a_{n-1})$ is a $k$-ary $n$-tuple,
then $\mathbf{a}$ is said to be \emph{symmetric} if and only if $a_i=a_{n-1-i}$ for every $i\in
\{0,1,\ldots,n-1\}$.  An $n$-tuple which is not symmetric is referred to as \emph{asymmetric}.
\end{definition}

We define two key subgraphs of the de Bruijn digraph.

\begin{definition}
Suppose $k\geq2$ and $n\geq1$.  Let $B_k^*(n)$ be the subgraph of the de Bruijn digraph $B_k(n)$
with all the edges corresponding to symmetric $(n+1)$-tuples removed.
\end{definition}

\begin{remark}  \label{remark:number_symmetric_tuples}
Since there are $k^{\lceil(n+1)/2\rceil}$ symmetric $(n+1)$-tuples, $B_k^*(n)$ contains
$k^{n+1}-k^{\lceil(n+1)/2\rceil}$ edges.
\end{remark}

\begin{definition}
Suppose $k\geq2$ and $n\geq2$.  Suppose $S$ is an $\mathcal{OS}_k(n)$.  Let $B(S,n)$ be the
subgraph of $B_k(n-1)$ with vertices the vertices of $B_k(n-1)$ and with edges corresponding to
those $n$-tuples which appear in either $S$ or $S^R$.  We refer to $B(S,n)$ as the
sequence-subgraph.
\end{definition}

The following simple lemma is key.

\begin{lemma}  \label{lemma:BSn_basic_properties}
Suppose $k\geq2$ and $n\geq2$.  Suppose $S$ is an $\mathcal{OS}_k(n)$ of period $m$.  Then:
\begin{itemize}
\item[i)] $B(S,n)$ contains $2m$ edges;
\item[ii)] every vertex of $B(S,n)$ has in-degree equal to its out-degree; and
\item[iii)] $B(S,n)$ is a subgraph of $B^*_k(n-1)$.
\end{itemize}
\end{lemma}

\begin{proof}
\begin{itemize}
\item[i)] Since $S$ has period $m$, there are a total of $2m$ $n$-tuples appearing in $S$ and
    $S^R$.  They are all distinct since $S$ is an $\mathcal{OS}_k(n)$.
\item[ii)] $S$ and $S^R$ correspond to edge-disjoint Eulerian circuits in $B(S,n)$, and the
    result follows.
\item[iii)] This is immediate since if $S$ is an $\mathcal{OS}_k(n)$, then neither $S$ nor
    $S^R$ can contain any symmetric $n$-tuples.
\end{itemize}
\end{proof}

The bound we give in this section derives from a careful analysis of the sequence-subgraph $B(S,n)$
of an $\mathcal{OS}_k(n)$.  In particular we consider the impact on the possible number of edges in
this subgraph arising from the constraints we give in Sections~\ref{subsection:in-out-constraints}
and \ref{subsection:parity-constraints} below. Any upper bound on the number of edges in $B(S,n)$
immediately gives rise to a bound on the period of an orientable sequence from
Lemma~\ref{lemma:BSn_basic_properties}(i).

\subsection{In-out-degree constraints on the sequence-subgraph}  \label{subsection:in-out-constraints}

In this section and the next we consider properties of the sequence-subgraph for $S$ that can be
derived from the assumption that $S$ is an $\mathcal{OS}_k(n)$.  We first need the following
definition.

\begin{definition}
Suppose $k\geq2$ and $n\geq3$.  An $n$-tuple $(a_0,a_1,\ldots,a_{n-1})$ is said to be
\emph{left-semi-symmetric} if $a_i=a_{n-i-2}$, $0\leq i\leq n-2$.  Equivalently,
$(a_0,a_1,\ldots,a_{n-1})$ is left-semi-symmetric if and only if $(a_0,a_1,\ldots,a_{n-2})$ is
symmetric.

Analogously, an $n$-tuple $(a_0,a_1,\ldots,a_{n-1})$ is said to be \emph{right-semi-symmetric} if
$a_i=a_{n-i}$, $1\leq i\leq n-1$.  Equivalently, $(a_0,a_1,\ldots,a_{n-1})$ is right-semi-symmetric
if and only if $(a_1,a_2,\ldots,a_{n-1})$ is symmetric.
\end{definition}

We can now state the following.

\begin{lemma}  \label{lemma:vertex_degree_k-1}
Suppose $k\geq2$ and $n\geq3$.  A vertex in $B^*_k(n)$ has:
\begin{itemize}
\item[i)] in-degree $k-1$ if and only if it its label is left-semi-symmetric; otherwise it has
    in-degree $k$;
\item[ii)] out-degree $k-1$ if and only if it its label is right-semi-symmetric; otherwise it
    has out-degree $k$.
\end{itemize}
\end{lemma}

\begin{proof}
For (i), the in-degree of every vertex in $B_k(n)$ is $k$.  However, if (and only if) an inbound
edge corresponds to a symmetric $(n+1)$-tuple, then this edge will not be in $B_k^*(n)$.  Such an
event can occur if and only if the vertex is labelled with a left-semi-symmetric $n$-tuple, and
there can only be one such symmetric inbound edge.  The result follows.  The proof of (ii) follows
using an exactly analogous argument.
\end{proof}

This immediately tells us that some edges in $B^*_k(n-1)$ cannot occur in $B(S,n)$ if $S$ is a
$\mathcal{OS}_k(n)$.  However, before describing exactly when this occurs, we first need the
following simple result.

\begin{lemma}  \label{lemma:combining_left_right-semi-symmetric}
Suppose $k\geq2$ and $n\geq4$.  Suppose the $(n-1)$-tuple $(a_0,a_1,\ldots,a_{n-2})$ is both
left-semi-symmetric and right-semi-symmetric.  Then
\begin{itemize}
\item[i)] if $n$ is even then $(a_0,a_1,\ldots,a_{n-2})$ is uniform;
\item[ii)] if $n$ is odd then $(a_0,a_1,\ldots,a_{n-2})$ is either uniform or alternating.
\end{itemize}
\end{lemma}

\begin{proof}
Suppose the $(n-1)$-tuple $(a_0,a_1,\ldots,a_{n-2})$ is both left-semi-symmetric and
right-semi-symmetric.  Then $a_i=a_{n-i-3}$, $0\leq i\leq n-3$, and $a_i=a_{n-i-1}$, $1\leq i\leq
n-2$.  Since $n\geq4$ this implies that there exist constants $c_0$ and $c_1$ such that
$c_0=a_{2i}$, $0\leq 2i\leq n-2$, and $c_1=a_{2j+1}$, $0\leq 2j+1\leq n-2$.  Hence (ii) follows.

If $n$ is even then we have $a_{(n-2)/2}=a_{(n-2)/2-1}$, and hence $c_0=c_1$ and (i) follows.
\end{proof}

The following result follows immediately from Lemmas~\ref{lemma:vertex_degree_k-1} and
\ref{lemma:combining_left_right-semi-symmetric}.

\begin{corollary}  \label{corollary:left-right-in-out}
Suppose $k\geq2$ and $n\geq4$. Consider a vertex in $B^*_k(n-1)$ with label
$\mathbf{a}=(a_0,a_1,\ldots,a_{n-2})$, where $\mathbf{a}$ is non-uniform.
\begin{itemize}
\item[i)] if $n$ is even and $\mathbf{a}$ is left-semi-symmetric, then its in-degree is $k-1$
    and its out-degree is $k$;
\item[i)] if $n$ is even and $\mathbf{a}$ is right-semi-symmetric, then its out-degree is $k-1$
    and its in-degree is $k$;
\item[iii)] if $n$ is odd and $\mathbf{a}$ is left-semi-symmetric and non-alternating, then its
    in-degree is $k-1$ and its out-degree is $k$;
\item[iv)] if $n$ is odd and $\mathbf{a}$ is right-semi-symmetric and non-alternating, then its
    out-degree is $k-1$ and its in-degree is $k$;
\end{itemize}
\end{corollary}

The above corollary immediately tells us that certain edges in $B^*_k(n-1)$ cannot occur in
$B(S,n)$, as follows.

\begin{corollary}  \label{corollary:excluded_tuples_by_unequal_degree}
Suppose $k\geq2$ and $n\geq4$ and $S$ is an $\mathcal{OS}_k(n)$. Then
\begin{itemize}
\item[i)] if $n$ is even, for every vertex in $B^*_k(n-1)$ corresponding to a non-uniform
    left-semi-symmetric $(n-1)$-tuple $(a_0,a_1,\ldots,a_{n-2})$, there is an edge
    $(a_0,a_1,\ldots,a_{n-2},x)$ in $B^*_k(n-1)$, for some $x$, that is not in $B(S,n)$;
\item[ii)] if $n$ is even, for every vertex in $B^*_k(n-1)$ corresponding to a non-uniform
    right-semi-symmetric $(n-1)$-tuple $(a_0,a_1,\ldots,a_{n-2})$, there is an edge
    $(y,a_0,a_1,\ldots,a_{n-2})$ in $B^*_k(n-1)$, for some $y$, that is not in $B(S,n)$;
\item[iii)] if $n$ is odd, for every vertex in $B^*_k(n-1)$ corresponding to a non-uniform
    non-alternating left-semi-symmetric $(n-1)$-tuple $(a_0,a_1,\ldots,a_{n-2})$, there is an
    edge $(a_0,a_1,\ldots,a_{n-2},x)$ in $B^*_k(n-1)$, for some $x$, that is not in $B(S,n)$;
\item[iv)] if $n$ is odd, for every vertex in $B^*_k(n-1)$ corresponding to a non-uniform
    non-alternating right-semi-symmetric $(n-1)$-tuple $(a_0,a_1,\ldots,a_{n-2})$, there is an
    edge $(y,a_0,a_1,\ldots,a_{n-2})$ in $B^*_k(n-1)$, for some $y$, that is not in $B(S,n)$.
\end{itemize}
\end{corollary}

\begin{proof}
The result follows immediately from Lemma~\ref{lemma:BSn_basic_properties}(ii) and
Corollary~\ref{corollary:left-right-in-out}.
\end{proof}

\subsection{Degree-parity constraints on the sequence-subgraph}  \label{subsection:parity-constraints}

We first give the following simple lemma.

\begin{lemma}  \label{lemma:semi-symmetric_not_symmetric}
Suppose $k\geq2$ and $n\geq2$.  If an $n$-tuple is both left-semi-symmetric and symmetric then it
is uniform.  Similarly, if an $n$-tuple is both right-semi-symmetric and symmetric then it is
uniform.
\end{lemma}

\begin{proof}
Let $\mathbf a=(a_0,\dots,a_{n-1})$ be an $n$-tuple which is both left-semi-symmetric and
symmetric. Then $a_i=a_{n-2-i}=a_{n-1-i}$ for $i=0,1,\dots,n-2$. It is immediate that $a_j=a_{j+1}$
for $j=0,1,\dots,n-2$ and thus $\mathbf a$ is uniform. The second claim follows by an analogous
argument.
\end{proof}

The following lemma is key.

\begin{lemma}  \label{lemma:symmetric_vertex_even_degree}
Suppose $k\geq2$ and $n\geq2$.  Suppose $S$ is an $\mathcal{OS}_k(n)$ and consider a vertex in
$B(S,n)$ with label $\mathbf{a}=(a_0,a_1,\ldots,a_{n-2})$, where $\mathbf{a}$ is symmetric.  Then
$\mathbf{a}$ has even in-degree and even out-degree in $B(S,n)$.
\end{lemma}

\begin{proof}
Both $S$ and $S^R$ correspond to an Eulerian circuit in $B(S,n)$, and these circuits are edge
disjoint and cover all the edges of $B(S,n)$. If $\mathbf{a}$ is symmetric then both circuits pass
through this vertex equally many times. It follows that $\mathbf{a}$ has even in-degree and even
out-degree in $B(S,n)$.
\end{proof}

We also need the following.

\begin{lemma}  \label{lemma:vertex-parity}
Suppose $k\geq2$ and $n\geq2$, and consider a vertex in $B^*_k(n-1)$ with label
$\mathbf{a}=(a_0,a_1,\ldots,a_{n-2})$.  Then:
\begin{itemize}
\item[i)] if $k$ is even and $\mathbf{a}$ is uniform than the vertex has odd in-degree in
    $B^*_k(n-1)$ (and out-degree);
\item[ii)] if $k$ is odd and $\mathbf{a}$ is symmetric and non-uniform then the vertex has odd
    in-degree in $B^*_k(n-1)$ (and out-degree).
\end{itemize}
\end{lemma}

\begin{proof}
\begin{itemize}
\item[i)] Since $a$ is uniform it is also left-semi-symmetric and right-semi-symmetric, and
    hence by Lemma~\ref{lemma:vertex_degree_k-1} it has in-degree and out-degree $k-1$.  Since
    $k$ is even the result follows.
\item[ii)] Since $\mathbf{a}$ is symmetric and non-uniform it cannot be left-semi-symmetric or
    right-semi-symmetric by Lemma~\ref{lemma:semi-symmetric_not_symmetric}.  Hence $\mathbf{a}$
    has in-degree and out-degree $k$ and the result follows since $k$ is odd.
\end{itemize}
\end{proof}

Combining Lemmas~\ref{lemma:symmetric_vertex_even_degree} and \ref{lemma:vertex-parity} immediately
gives the following important result.

\begin{corollary}  \label{corollary:excluded-tuples-by-parity}
Suppose $k\geq2$ and $n\geq2$ and $S$ is an $\mathcal{OS}_k(n)$. Then
\begin{itemize}
\item[i)] if $k$ is even, for every vertex in $B^*_k(n-1)$ corresponding to a uniform
    $(n-1)$-tuple $(a,a,...,a)$, there is an edge $(a,a,...,a,x)$ in $B^*_k(n-1)$, for some
    $x$, that is not in $B(S,n)$.
\item[ii)] if $k$ is even, for every vertex in $B^*_k(n-1)$ corresponding to a uniform
    $(n-1)$-tuple $(a,a,...,a)$, there is an edge $(y,a,a,...,a)$ in $B^*_k(n-1)$, for some
    $y$, that is not in $B(S,n)$.
\item[iii)] if $k$ is odd, for every vertex in $B^*_k(n-1)$ corresponding to a symmetric and
    non-uniform $(n-1)$-tuple $(a_0,a_1,\ldots,a_{n-2})$, there is an edge
    $(a_0,a_1,\ldots,a_{n-2},x)$ in $B^*_k(n-1)$, for some $x$, that is not in $B(S,n)$.
\item[iv)] if $k$ is odd, for every vertex in $B^*_k(n-1)$ corresponding to a symmetric and
    non-uniform $(n-1)$-tuple $(a_0,a_1,\ldots,a_{n-2})$, there is an edge
    $(y,a_0,a_1,\ldots,a_{n-2})$ in $B^*_k(n-1)$, for some $y$, that is not in $B(S,n)$.
\end{itemize}
\end{corollary}

\begin{remark}
Observe that, for $n=2$ and $k$ odd, Lemma~\ref{lemma:vertex-parity} and
Corollary~\ref{corollary:excluded-tuples-by-parity} are trivially true since there are no symmetric
non-uniform $(n-1)$-tuples.
\end{remark}

\subsection{Constraint interactions}  \label{subsection:constraint-interactions}

The results above indicate in what circumstances certain edges in $B^*_k(n-1)$ cannot occur in
$B(S,n)$ if $S$ is an $\mathcal{OS}_k(n)$.  In particular we have considered cases where one of the
incoming edges to a vertex corresponding to a symmetric or right-semi-symmetric $(n-1)$-tuple
cannot occur in $B(S,n)$, and also where one of the outgoing edges from a vertex corresponding to a
symmetric or left-semi-symmetric $(n-1)$-tuple cannot occur in $B(S,n)$. While we would like to add
together the numbers of eliminated edges for each case, we need to ensure we avoid `double
counting'.

More specifically we need to consider when an edge can be outgoing from a symmetric or
left-semi-symmetric $(n-1)$-tuple and also incoming to a symmetric or right-semi-symmetric
$(n-1)$-tuple.  This motivates the following result.

\begin{lemma}  \label{lemma:constraint-interactions}
Suppose $k\geq2$.  Suppose $\mathbf{a}=(a_0,a_1,\dots,a_{n-2})$ and
$\mathbf{b}=(a_1,a_2,\dots,a_{n-1})$ are $k$-ary $(n-1)$-tuples, connected by the edge
$(a_0,a_1,\dots,a_{n-1})$.
\begin{itemize}
\item[i)] If $n\geq3$, $\mathbf{a}$ is symmetric and $\mathbf{b}$ is symmetric then either
    $\mathbf{a}$ and $\mathbf{b}$ are both uniform or $n$ is even and $\mathbf{a}$ and
    $\mathbf{b}$ are both alternating.
\item[ii)] If $n\geq5$, $\mathbf{a}$ is symmetric and $\mathbf{b}$ is right-semi-symmetric then
    there exist $c_j$, $0\leq j\leq2$, such that $a_{3i+j}=c_j$, for every $i$ and $0\leq j\leq
    2$.  Moreover if $n\equiv0\pmod3$ then $c_0=c_1$, if $n\equiv1\pmod3$ then $c_0=c_2$, and
    if $n\equiv2\pmod3$ then $c_1=c_2$.
\item[iii)] If $n\geq5$, $\mathbf{a}$ is left-semi-symmetric and $\mathbf{b}$ is symmetric then
    there exist $c_j$, $0\leq j\leq2$, such that $a_{3i+j}=c_j$, for every $i$ and $0\leq j\leq
    2$.  Moreover if $n\equiv0\pmod3$ then $c_1=c_2$, if $n\equiv1\pmod3$ then $c_0=c_1$, and
    if $n\equiv2\pmod3$ then $c_0=c_2$.
\item[iv)] If $n\geq5$, $\mathbf{a}$ is left-semi-symmetric and $\mathbf{b}$ is
    right-semi-symmetric then there exist $c_j$, $0\leq j\leq3$, such that $a_{4i+j}=c_j$, for
    every $i$ and $0\leq j\leq 3$. Moreover if $n\equiv0\pmod4$ then $c_0=c_1$ and $c_2=c_3$,
    if $n\equiv1\pmod4$ then $c_0=c_2$, if $n\equiv2\pmod4$ then $c_0=c_3$ and $c_1=c_2$, and
    if $n\equiv3\pmod4$ then $c_1=c_3$.
\end{itemize}
\end{lemma}

\begin{proof}
\begin{itemize}
\item[i)] By symmetry of $\mathbf{a}$ and $\mathbf{b}$, respectively, $a_i=a_{n-2-i}$ and
    $a_{i+1}=a_{n-1-i}$ for every $i$, $0\leq i\leq n-2$.  Hence there exist $c_0$ and $c_1$
    such that $a_{2i}=c_0$ and $a_{2i+1}=c_1$ for every $i$.  If $n$ is odd then $c_0=c_1$ and
    the result follows.
\item[ii)] By symmetry of $\mathbf{a}$, $a_i=a_{n-2-i}$ for every $i$ ($0\leq i\leq n-2$), and
    by right-semi-symmetry of $\mathbf{b}$, $a_{j}=a_{n+1-j}$ for every $j$ ($2\leq j\leq
    n-1$). Hence $a_i=a_{i+3}$, $0\leq i< n-3$.  Now, since $n\geq5$, $a_{3i+j}=c_j$, for every
    $i$ and $0\leq j\leq 2$ for some $c_j$.  From symmetry of $\mathbf{a}$, we have
    $c_0=a_0=a_{n-2}$; if $n\equiv0\pmod3$ then $a_{n-2}=c_1$, i.e.\ $c_0=c_1$. From
    right-semi-symmetry of $\mathbf{b}$, $c_2=a_2=a_{n-1}$; if $n\equiv1\pmod3$ then
    $a_{n-1}=c_0$, and so $c_0=c_2$, and finally if $n\equiv2\pmod3$ then $a_{n-1}=c_1$, i.e.\
    $c_1=c_2$.
\item[iii)] By left-semi-symmetry of $\mathbf{a}$, $a_i=a_{n-3-i}$ for every $i$ ($0\leq i\leq
    n-3$), and by symmetry of $\mathbf{b}$, $a_j=a_{n-j}$ for every $j$ ($1\leq j\leq n-1$).
    Using exactly the same argument as for (ii), $a_{3i+j}=c_j$, for every $i$ and $0\leq j\leq
    2$ for some $c_j$. By the symmetry of $\mathbf{b}$, $c_1=a_1=a_{n-1}$; if $n\equiv0\pmod3$
    then $a_{n-1}=c_2$ (and so $c_1=c_2$), and if $n\equiv1\pmod3$ then $a_{n-1}=c_0$ and so
    $c_1=c_0$.  Finally, by left-semi-symmetry of $\mathbf{a}$, $c_0=a_0=a_{n-3}$, and hence if
    $n\equiv2\pmod3$ then $a_{n-3}=c_2$ and so $c_0=c_2$.
\item[iv)] By left-semi-symmetry of $\mathbf{a}$,  $a_i=a_{n-3-i}$ for every $i$ ($0\leq i\leq
    n-3$) and by right-semi-symmetry of $\mathbf{b}$, $a_{j}=a_{n+1-j}$ for every $j$ ($2\leq
    j\leq n-1$). Hence $a_i=a_{i+4}$, $0\leq i\leq n-3$.  Thus, since $n\geq5$, $a_{4i+j}=c_j$,
    for every $i$ and $0\leq j\leq 3$ for some $c_j$. From left-semi-symmetry of $\mathbf{a}$,
    $c_0=a_0=a_{n-3}$; if $n\equiv0\pmod4$ then $a_{n-3}=c_1$ and so $c_0=c_1$, if
    $n\equiv1\pmod4$ then $a_{n-3}=c_2$ and so $c_0=c_2$, and if $n\equiv2\pmod4$ then
    $a_{n-3}=c_3$ and so $c_0=c_3$. Further, we have $c_1=a_1=a_{n-4}$; if $n\equiv3\pmod4$
    then $a_{n-4}=c_3$ and so $c_1=c_3$. From right-semi-symmetry of $\mathbf{b}$, we have
    $c_2=a_2=a_{n-1}$; if $n\equiv0\pmod4$ then $a_{n-1}=c_3$ (and so $c_2=c_3$), and finally
    if $n\equiv2\pmod4$ then $a_{n-1}=c_1$, i.e.\ $c_1=c_2$.
\end{itemize}
\end{proof}


\subsection{Counting special types of $n$-tuple}

The following simple enumerative result will be of use below.

\begin{lemma}  \label{lemma:tuple_numbers}
Suppose $n\geq3$ and $k\geq2$.  Then
\begin{itemize}
\item[i)] the number of symmetric $k$-ary $n$-tuples is $k^{\lceil n/2\rceil}$;
\item[ii)] the number of uniform $k$-ary $n$-tuples is $k$.
\item[iii)] the number of symmetric non-uniform $k$-ary $n$-tuples is $k^{\lceil n/2\rceil}-k$;
\item[iv)] the number of asymmetric $k$-ary $n$-tuples is $k^n-k^{\lceil n/2\rceil}$;
\item[v)] the number of alternating $k$-ary $n$-tuples is $k(k-1)$;
\item[vi)]  the number of non-uniform non-alternating symmetric $k$-ary $n$-tuples is
    $k^{(n+1)/2}-k^2$ if $n$ is odd, and $k^{n/2}-k$ if $n$ is even;
\item[vii)] the number of left-semi-symmetric $k$-ary $n$-tuples is $k^{\lceil(n+1)/2\rceil}$;
\item[viii)] the number of non-uniform left-semi-symmetric $k$-ary $n$-tuples is
    $k^{\lceil(n+1)/2\rceil}-k$;
\item[ix)] the number of non-uniform non-alternating left-semi-symmetric $k$-ary $n$-tuples is
    $k^{(n+2)/2}-k^2$ if $n$ is even, and $k^{(n+1)/2}-k$ if $n$ is odd.
\end{itemize}
\end{lemma}

\begin{proof}
Result (i) follows since there are $k$ choices for each of the first $\lceil n/2\rceil$ positions
of an $n$-tuple.  (ii) is immediate, and (iii) is simply the difference between (i) and (ii). (iv)
follows from (i) and the trivial observation that there are $k^n$ possible $k$-ary $n$-tuples. (v)
is immediate by observing that there are $k$ choices for $c_0$ and $k-1$ choices for $c_1$, given
that the $n$-tuple is non-uniform.

For (vi), if $n$ is odd then every alternating $n$-tuple is symmetric, and hence the number of
non-uniform non-alternating symmetric $n$-tuples is the number of non-uniform symmetric $n$-tuples
less the number of alternating $n$-tuples (which are non-uniform by definition), i.e.\ the
difference between (iii) and (v).  If $n$ is even, then an alternating $n$-tuple cannot be
symmetric, and hence the number of non-uniform non-alternating symmetric $n$-tuples is the number
of non-uniform symmetric $n$-tuples, i.e.\ (iii).

(vii) follows from the fact that the first $\lceil(n-1)/2\rceil$ positions of the $n$-tuple and the
last position can be chosen freely, and the other positions are determined. (viii) is simply the
difference between (vii) and (ii), observing that a uniform $n$-tuple is left-semi-symmetric.

Finally, for (ix), if $n$ is even then every alternating $n$-tuple is left-semi-symmetric, and
hence the number of non-uniform non-alternating left-semi-symmetric $n$-tuples is the number of
non-uniform left-semi-symmetric $n$-tuples less the number of alternating $n$-tuples (which are
non-uniform by definition), i.e.\ the difference between (viii) and (v).  If $n$ is odd, then an
alternating $n$-tuple cannot be left-semi-symmetric, and hence the number of non-uniform
non-alternating left-semi-symmetric $n$-tuples is the number of non-uniform left-semi-symmetric
$n$-tuples, i.e.\ (viii).
\end{proof}

\subsection{Developing the bound}

Before establishing our new period bound, we first outline the proof strategy.  This will then
enable the proof of the bound to be described more simply.

In Sections~\ref{subsection:in-out-constraints} and \ref{subsection:parity-constraints} we
described two main ways in which we could establish that certain edges in $B^*_k(n-1)$ cannot occur
in $B(S,n)$ when $S$ is an $\mathcal{OS}_k(n)$.  In particular we showed that in two cases incoming
edges to certain categories of vertex cannot occur, and also in two cases that outgoing edges from
certain categories of vertex cannot occur.

This suggests a straightforward strategy for bounding the period of an $\mathcal{OS}_k(n)$, namely
that the period is at most half the maximum cardinality of $B(S,n)$ (by
Lemma~\ref{lemma:BSn_basic_properties}(i)).  In turn $|B(S,n)|$ is bounded above by the number of
edges in $B^*_k(n-1)$, i.e.\ $k^n-k^{\lceil n/2\rceil}$ by Lemma~\ref{lemma:tuple_numbers}(iv),
less the number of edges in $B^*_k(n-1)$ that cannot occur in $B(S,n)$ as specified in
Corollaries~\ref{corollary:excluded_tuples_by_unequal_degree} and
\ref{corollary:excluded-tuples-by-parity}.

This strategy is complicated by the fact that, as noted in
Section~\ref{subsection:constraint-interactions}, there is a danger of `double counting' certain
excluded edges.  For example, Corollary~\ref{corollary:excluded_tuples_by_unequal_degree} asserts
that certain outgoing edges from a left-semi-symmetric $(n-1)$-tuple cannot occur, and
Corollary~\ref{corollary:excluded-tuples-by-parity} asserts that certain edges incoming to a
uniform or symmetric $(n-1)$ tuple cannot occur.  These two sets of excluded edges may overlap, and
hence we need to take this into account; Lemma~\ref{lemma:constraint-interactions} is of key
importance in this respect.

The following notation is intended to simplify the arguments.  Let $U_{\text{in}}$ and
$U_{\text{out}}$ be the sets of incoming and outgoing edges excluded by
Corollary~\ref{corollary:excluded_tuples_by_unequal_degree}, and $P_{\text{in}}$ and
$P_{\text{out}}$ be the sets of incoming and outgoing edges excluded by
Corollary~\ref{corollary:excluded-tuples-by-parity}.

The above discussion leads to the following key lemma.

\begin{lemma}  \label{lemma:bounds-high-level}
If $k\geq2$ and $S$ is an $\mathcal{OS}_k(n)$ then:
\[ |B(S,n)|\leq \begin{cases}
k^2-k - |P_{\text{out}}|,
 & \text{if $n=2$} \\
k^3-k^2 - |P_{\text{out}}|-|P_{\text{in}}| + |P_{\text{out}}\cap P_{\text{in}}|,
 & \text{if $n=3$} \\
k^4-k^2 - |U_{\text{out}}|-|P_{\text{out}}|,
 & \text{if $n=4$} \\
k^n-k^{\lceil n/2\rceil} - |U_{\text{out}}|-|U_{\text{in}}|-|P_{\text{out}}|-|P_{\text{in}}| \\
~~~~+~|U_{\text{out}}\cap P_{\text{in}}|+|P_{\text{out}}\cap U_{\text{in}}| \\
~~~~+~|U_{\text{out}}\cap U_{\text{in}}|+|P_{\text{out}}\cap P_{\text{in}}|,
 & \text{if $n>4$.}
\end{cases} \]
where $|X\cap Y|$ denotes the maximum possible cardinality for such a set.
\end{lemma}

\begin{proof}
The argument for $n=2$ is immediate. A similar comment applies when $n=3$.  The $n=4$ and $n \ge 5$
cases follow from observing that $U_{\text{in}}\cap P_{\text{in}}=U_{\text{out}}\cap
P_{\text{out}}=\emptyset$ because a symmetric $(n-1)$-tuple can neither be non-uniform
left-semi-symmetric nor non-uniform right-semi-symmetric, from
Lemma~\ref{lemma:semi-symmetric_not_symmetric}.
\end{proof}

\begin{remark}
The reason to restrict the sets considered when $n\leq 4$ is because certain sets are empty or
might be equal for small $n$.  Moreover Corollary~\ref{corollary:excluded_tuples_by_unequal_degree}
only applies for $n\geq4$.
\end{remark}

\subsection{The bound}

We can now give the bound.

\begin{theorem} \label{theorem:OS_bound_4}
Suppose $k\geq2$ and $m$ is the period of an $\mathcal{OS}_k(n)$.
\begin{itemize}
\item[i)] If $n=2$ then:
\[ m \leq \begin{cases}
\frac{k^2-k}{2}, & \text{if $k$ is odd} \\
\frac{k^2-2k}{2}, & \text{if $k$ is even}
\end{cases} \]
\item[ii)] If $n=3$ then:
\[ m \leq \begin{cases}
\frac{k^3-k^2}{2}, & \text{if $k$ is odd} \\
\frac{k^3-k^2-2k}{2}, & \text{if $k$ is even}
\end{cases} \]
\item[iii)] If $n=4$ then:
\[ m \leq \begin{cases}
\frac{k^4-3k^2+2k}{2}, & \text{if $k$ is odd} \\
\frac{k^4-2k^2}{2}, & \text{if $k$ is even}
\end{cases}\]
\item[iv)] If $n>4$ then:
\[ m \leq \begin{cases}
\frac{k^n-3k^{(n+1)/2}-2k^{(n-1)/2}+k^3+3k^2}{2}, & \text{if $n$ and $k$ are odd} \\
\frac{k^n-3k^{(n+1)/2}+k^3+k^2-2k}{2}, & \text{if $n$ is odd and $k$ is even} \\
\frac{k^n-5k^{n/2}+4k^2}{2}, & \text{if $n$ is even and $k$ is odd} \\
\frac{k^n-3k^{n/2}+k^2-k}{2}, & \text{if $n$ and $k$ are even}
\end{cases}\]

\end{itemize}

\end{theorem}

\begin{proof}
The proof builds on, and uses the notation of, Lemma~\ref{lemma:bounds-high-level}.
\begin{itemize}

\item[i)] {\bf $n=2$}.  If $k$ is odd then $P_{\text{out}}=\emptyset$ by
    Corollary~\ref{corollary:excluded-tuples-by-parity}(iii), since there are no symmetric
    non-uniform 1-tuples.  If $k$ is even then, by
    Corollary~\ref{corollary:excluded-tuples-by-parity}(i), $|P_{\text{out}}|=k$, since there
    are $k$ uniform 1-tuples by Lemma~\ref{lemma:tuple_numbers}(ii). The result follows from
    Lemma~\ref{lemma:bounds-high-level}.

\item[ii)] {\bf $n=3$}.  If $k$ is odd then, by
    Corollary~\ref{corollary:excluded-tuples-by-parity}(iii),(iv),
    $|P_{\text{in}}|=|P_{\text{out}}|$ equals the number of symmetric non-uniform 2-tuples,
    i.e.\ zero by Lemma~\ref{lemma:tuple_numbers}(iii).

    If $k$ is even then, by Corollary~\ref{corollary:excluded-tuples-by-parity}(i),(ii),
    $|P_{\text{in}}|=|P_{\text{out}}|$ equals the number of uniform 2-tuples, i.e.\ $k$ by
    Lemma~\ref{lemma:tuple_numbers}(ii). Moreover, $P_{\text{in}}\cap
    P_{\text{out}}=\emptyset$, since an edge can only be outgoing from a uniform ($n-1)$-tuple
    and incoming to a uniform ($n-1)$-tuple if it corresponds to a uniform $n$-tuple, and such
    an edge cannot occur in $B^*_k(n-1)$. The result follows from
    Lemma~\ref{lemma:bounds-high-level}.

\item[iii)] {\bf $n=4$}. By Corollary~\ref{corollary:excluded_tuples_by_unequal_degree}(i),
    $|U_{\text{out}}|$ equals the number of non-uniform left-semi-symmetric $3$-tuples, i.e.\
    $k^2-k$ by Lemma~\ref{lemma:tuple_numbers}(viii).  If $k$ is odd then, by
    Corollary~\ref{corollary:excluded-tuples-by-parity} (iii), $|P_{\text{out}}|$ equals the
    number of symmetric non-uniform 3-tuples, i.e.\ $k^2-k$ by
    Lemma~\ref{lemma:tuple_numbers}(iii). If $k$ is even then, by
    Corollary~\ref{corollary:excluded-tuples-by-parity} (i), $|P_{\text{out}}|$ equals the
    number of uniform 3-tuples, i.e.\ $k$ by Lemma~\ref{lemma:tuple_numbers}(ii). The result
    follows from Lemma~\ref{lemma:bounds-high-level}.

\item[iv)a)] {\bf $n>4$; $n$ odd}.  By
    Corollary~\ref{corollary:excluded_tuples_by_unequal_degree}(iii),(iv),
    $|U_{\text{out}}|=|U_{\text{in}}|$ equals the number of non-uniform non-alternating
    left-semi-symmetric $(n-1)$-tuples, i.e.\ $k^{(n+1)/2}-k^2$ by
    Lemma~\ref{lemma:tuple_numbers}(ix).  Also, by
    Lemma~\ref{lemma:constraint-interactions}(iv), $|U_{\text{out}}\cap U_{\text{in}}|=k^3-k^2$
    since if, in the statement of the lemma, we choose $c_1\neq c_3$ if $n\equiv 1\pmod 4$, and
    $c_0\neq c_2$ if $n\equiv 3\pmod 4$, then $\mathbf{a}$ and $\mathbf{b}$ are non-uniform and
    non-alternating and there are $k^2(k-1)=k^3-k^2$ possibilities for the values $c_i$.

    If $k$ is odd then, by Corollary~\ref{corollary:excluded-tuples-by-parity} (iii),(iv),
    $|P_{\text{out}}|=|P_{\text{in}}|$ equals the number of symmetric non-uniform
    $(n-1)$-tuples, i.e.\ $k^{(n-1)/2}-k$ by Lemma~\ref{lemma:tuple_numbers}(iii).\newline By
    Lemma~\ref{lemma:constraint-interactions}(ii),(iii), $|U_{\text{out}}\cap
    P_{\text{in}}|=|P_{\text{out}}\cap U_{\text{in}}|=k(k-1)$ (removing the case where the
    $c_i$ are all equal).
    \newline By Lemma~\ref{lemma:constraint-interactions}(i), $P_{\text{out}}\cap
    P_{\text{in}}=\emptyset$.

    If $k$ is even then, by Corollary~\ref{corollary:excluded-tuples-by-parity}(i),(ii),
    $|P_{\text{out}}|=|P_{\text{in}}|$ equals the number of uniform $(n-1)$-tuples, i.e.\ $k$
    by Lemma~\ref{lemma:tuple_numbers}(ii).
    \newline By Corollary~\ref{corollary:excluded-tuples-by-parity}(i),(ii), $P_{\text{out}}$
    and $P_{\text{in}}$ only contain edges out-going/in-going from/to uniform $(n-1)$-tuples,
    and by Corollary~\ref{corollary:excluded_tuples_by_unequal_degree}(iii),(iv)
    $U_{\text{out}}$ and $U_{\text{in}}$ only contain edges that are out-going/in-going from/to
    non-uniform $(n-1)$-tuples, and hence $|U_{\text{out}}\cap
    P_{\text{in}}|=|P_{\text{out}}\cap U_{\text{in}}|=\emptyset$.
    \newline By Corollary~\ref{corollary:excluded-tuples-by-parity}(i),(ii), $P_{\text{out}}$
    and $P_{\text{in}}$ only contain edges out-going/in-going from/to uniform $(n-1)$-tuples,
    and hence an edge in $P_{\text{out}}\cap P_{\text{in}}$ must be uniform, i.e.\
    $P_{\text{out}}\cap P_{\text{in}}=\emptyset$.

    The result follows from Lemma~\ref{lemma:bounds-high-level}.

\item[iv)b)] {\bf $n>4$; $n$ even}. By
    Corollary~\ref{corollary:excluded_tuples_by_unequal_degree}(i),(ii),
    $|U_{\text{out}}|=|U_{\text{in}}|$ equals the number of non-uniform left-semi-symmetric
    $(n-1)$-tuples, i.e.\ $k^{n/2}-k$ by Lemma~\ref{lemma:tuple_numbers}(viii). Additionally,
    by Lemma~\ref{lemma:constraint-interactions}(iv), $|U_{\text{out}}\cap
    U_{\text{in}}|=k(k-1)$, by choosing $c_0$ and $c_2$ to be distinct.

    If $k$ is odd then, by Corollary~\ref{corollary:excluded-tuples-by-parity}(iii),(iv),
    $|P_{\text{out}}|=|P_{\text{in}}|$ equals the number of symmetric non-uniform
    $(n-1)$-tuples, i.e.\ $k^{n/2}-k$ by Lemma~\ref{lemma:tuple_numbers}(iii).
    \newline By Lemma~\ref{lemma:constraint-interactions}(ii),(iii), $|U_{\text{out}}\cap
    P_{\text{in}}|=|P_{\text{out}}\cap U_{\text{in}}|=k(k-1)$, by ensuring that $c_0$, $c_1$
    and $c_2$ are not all equal.
    \newline By Lemma~\ref{lemma:constraint-interactions}(i), $|P_{\text{out}}\cap P_{\text{in}}|=k(k-1)$,
    i.e.\ the number of non-uniform alternating $(n-1)$-tuples.

    If $k$ is even then, by Corollary~\ref{corollary:excluded-tuples-by-parity}(i),(ii),
    $|P_{\text{out}}=|P_{\text{in}}|$ equals the number of uniform $(n-1)$-tuples, i.e.\ $k$ by
    Lemma~\ref{lemma:tuple_numbers}(i).
    \newline By Corollary~\ref{corollary:excluded-tuples-by-parity}(i),(ii), $P_{\text{out}}$
    and $P_{\text{in}}$ only contain edges out-going/in-going from/to uniform $(n-1)$-tuples,
    and by Corollary~\ref{corollary:excluded_tuples_by_unequal_degree}(i),(ii) $U_{\text{out}}$
    and $U_{\text{in}}$ only contain edges that are out-going/in-going from/to non-uniform
    $(n-1)$-tuples, and hence $|U_{\text{out}}\cap P_{\text{in}}|=|P_{\text{out}}\cap
    U_{\text{in}}|=\emptyset$.
    \newline Finally, it is immediate that $P_{\text{out}}\cap P_{\text{in}}=\emptyset$.

    The result follows from Lemma~\ref{lemma:bounds-high-level}.

\end{itemize}
\end{proof}

\subsection{Numerical results}

The bounds resulting from Theorem~\ref{theorem:OS_bound_4} are tabulated for small $k$ and $n$ in
Table~\ref{table:new_OS_periods_bounds}.  The numbers given in brackets are the bounds derived from
\cite[Theorem 4.11]{Alhakim24a}, and are provided for comparison purposes.

\begin{table}[htb]
\centering \caption{Bounds --- new and (old) --- on the period of an $\mathcal{OS}_k(n)$}
\label{table:new_OS_periods_bounds}
\begin{footnotesize}
\begin{tabular}{crrrrrrr} \hline
$n$ & $k=2$   & $k=3$ & $k=4$   & $k=5$   & $k=6$    & $k=7$     & $k=8$     \\ \hline
2   & 0       & 3     & 4       & 10      & 12       & 21        & 24        \\
    & (0)     & (3)   & (4)     & (10)    & (12)     & (21)      & (24)      \\ \hline
3   & 0       & 9     & 20      & 50      & 84       & 147       & 216       \\
    & (1)     & (9)   & (22)    & (50)    & (87)     & (147)     & (220)     \\ \hline
4   & 4       & 30    & 112     & 280     & 612      & 1134      & 1984      \\
    & (5)     & (33)  & (118)   & (290)   & (627)    & (1155)    & (2012)    \\ \hline
5   & 8       & 99    & 452     & 1450    & 3684     & 8085      & 15896     \\
    & (11)    & (105) & (478)   & (1490)  & (3777)   & (8211)    & (16124)   \\ \hline
6   & 21      & 315   & 1958    & 7550    & 23019    & 58065     & 130332    \\
    & (27)    & (336) & (2014)  & (7680)  & (23217)  & (58464)   & (130812)  \\ \hline
7   & 44      & 972   & 7844    & 38100   & 138144   & 408072    & 1042712   \\
    & (55)    & (1032)& (8062)  & (38640) & (139317) & (410256)  & (1046524) \\ \hline
8   & 105     & 3096  & 32390   & 193800  & 837879   & 2876496   & 8382492   \\
    & (119)   & (3189)& (32638) & (194630)& (839157) & (2879835) & (8386556) \\ \hline
9   & 212     & 9423  & 129572  & 971350  & 5027304  & 20149437  & 67059992  \\
    & (239)   & (9645)& (130558)& (974390)& (5034957)& (20166027)& (67092476)\\ \hline
\end{tabular}
\end{footnotesize}
\end{table}

\color{black}

\section{Eulerian cycles in the de Bruijn digraph}  \label{section:Eulerian_cycles}

\subsection{Overview}

Inspired by the well-known correspondence between $k$-ary de Bruijn sequences of order $n$ and
Eulerian cycles in the order $n-1$ de Bruijn graph $B_{k}(n-1)$ (see, for example,
\cite{Lempel70}), we next describe a correspondence between orientable sequences of order $n$ and
Eulerian cycles in subgraphs of the de Bruijn graph $B_{k}(n-1)$, where these subgraphs satisfy
certain special properties.  If we can then construct subgraphs with these special properties which
admit Eulerian cycles then we have a simple method of generating orientable sequences.

\subsection{A correspondence}

First observe that the set of directed edges in any subgraph of $B_{k}(n)$ corresponds to a subset
of the set $\mathbb{Z}_k^{n+1}$, i.e.\ a subset of all possible $k$-ary $(n+1)$-tuples.  We need
the following definitions.

The following is standard terminology.

\begin{definition}  \label{definition;Eulerian}
An \emph{Eulerian digraph} is a connected digraph for which every vertex has in-degree equal to out-degree.
\end{definition}

The name derives from the fact that there exists an Eulerian circuit, i.e.\ a circuit
visiting every edge,  in a digraph if and only if the digraph is Eulerian --- see, for example, Corollary 6.1 of
Gibbons \cite{Gibbons85}.  Moreover, there are simple and efficient algorithms for finding
Eulerian
circuits --- see for example Gibbons~\cite[Figure 6.5]{Gibbons85}.

We next define a special class of digraphs of importance for this paper.

\begin{definition}  \label{definition:antisymmetric}
Suppose $T$ is a subgraph of the de Bruijn digraph $B_{k}(n)$ for some $n\geq2$ and $k\geq2$.  $T$
is said to be \emph{antisymmetric} if the following property holds.

Suppose $\mathbf{x}=(x_0,x_1,\ldots,x_{n-1}),\mathbf{y}=(y_0,y_1,\ldots,y_{n-1})\in
\mathbb{Z}_k^n$, i.e.\ they are vertices in $T$.  Then if $(\mathbf{x},\mathbf{y})$ is an edge in
$T$, then $(\mathbf{y}^R,\mathbf{x}^R)$ is not an edge in $T$.
\end{definition}

\begin{definition}  \label{definition:En(s)}
Suppose $S$ is a periodic $k$-ary $n$-window sequence of period $m$ for some $n\geq2$, $k\geq2$
and
$m\geq1$. Then the \emph{edge-graph} $E_n(S)$ of $S$ is defined to be the subgraph of $B_{k}(n-1)$
whose directed edges correspond to $n$-tuples appearing in a period of $S$, i.e.\ with edge set
\[ \{ \mathbf{s}_n(i)~:~0\leq i\leq m-1  \} \]
and whose vertices are those vertices of $B_{k}(n-1)$ that have in-degree at least one.
\end{definition}

We can now state a key lemma.

\begin{lemma}  \label{lemma:OS_iff_antisymmetric}
Suppose $S$ is a $k$-ary periodic $n$-window sequence.  Then $S$ is an $\mathcal{OS}_k(n)$ if and
only if $E_n(S)$ is antisymmetric in $B_{k}(n-1)$.
\end{lemma}

\begin{proof}
Suppose $S$ is an $\mathcal{OS}_k(n)$.  Then, by definition, $\mathbf{s}_n(i) \neq
\mathbf{s}_n(j)^R$, for any $i,j$.  Hence, again by definition, this means that $\mathbf{a}\neq
\mathbf{b}^R$ for any edges $\mathbf{a},\mathbf{b}$ in $E_n(S)$.  Hence $E_n(S)$ is antisymmetric.

Now suppose $E_n(S)$ is antisymmetric, and hence $\mathbf{a}\neq \mathbf{b}^R$ for any edges
$\mathbf{a},\mathbf{b}$ in $E_n(S)$.  Thus, $\mathbf{s}_n(i) \neq \mathbf{s}_n(j)^R$, for any
$i,j$, and so $S$ is an $\mathcal{OS}_k(n)$.
\end{proof}

This enables us to give our main result.

\begin{theorem}  \label{theorem:correspondence}
If $S$ is an $\mathcal{OS}_k(n)$ of period $m$ then $E_n(S)$ is an antisymmetric Eulerian subgraph
of $B_{k}(n-1)$ containing $m$ edges.  Moreover, if $T$ is an antisymmetric Eulerian subgraph of
$B_{k}(n-1)$ with $m$ edges, then there exists an $\mathcal{OS}_k(n)$ $S$ of period $m$ such that
$E_n(S)=T$.
\end{theorem}

\begin{proof}
Suppose $S=(s_i)$ is an $\mathcal{OS}_k(n)$ of period $m$.  Then, by
Lemma~\ref{lemma:OS_iff_antisymmetric}, $E_n(S)$ is antisymmetric in $B_{k}(n-1)$.  Also,
trivially, $E_n(S)$ contains $m$ edges.  Finally, if $(x_1,x_2,\ldots,x_{n-1})$ is a vertex in
$E_n(S)$, then the incoming edges all have labels of the form $(x,x_1,x_2,\ldots,x_{n-1})$ for
some
$x$, and correspond to $\mathbf{s}_n(i)$ for some $i$.  Then the edge corresponding to
$\mathbf{s}_n(i+1)$ will have a label of the form $(x_1,x_2,\ldots,x_{n-1},y)$ for some $y$, and
so
for every incoming edge there is an outgoing edge, and vice versa.  Moreover, by the definition of
an orientable sequence, distinct values of $x$ correspond to distinct values of $y$. Hence the
in-degree of every vertex is the same as the out-degree.  Finally, note that $E_n(S)$ is connected
since we only include vertices in $E_n(S)$ with in-degree greater than zero, and the sequence $S$
defines a path incorporating every vertex in $E_n(S)$.

Now suppose $T$ is an antisymmetric Eulerian subgraph of $B_{k}(n-1)$ with $m$ edges.  Then there
exists an Eulerian circuit of length $m$ in $B_{k}(n-1)$.  This circuit corresponds to a sequence
$S$ with period $m$ in the natural way.  $S$ is clearly an $n$-window sequence since the circuit
visits every edge exactly once, and each edge corresponds to a unique $n$-tuple.  It is
straightforward to verify that $E_n(S)=T$. Finally, $S$ is an $\mathcal{OS}_k(n)$ from
Lemma~\ref{lemma:OS_iff_antisymmetric}.
\end{proof}

\begin{remark}
In general, for any antisymmetric Eulerian subgraph of $B_{k}(n-1)$ $T$, there will exist many
orientable sequences with edge-graph $T$, since there may be many different Eulerian circuits in
$T$, each corresponding to a different sequence.
\end{remark}

\subsection{Implications}

Theorem~\ref{theorem:correspondence} means that if we can construct an antisymmetric Eulerian
subgraph of $B_{k}(n-1)$ with $m$ edges for `large' $m$ (i.e.\ with $m$ close to the maximum
possible), then we will immediately have a set of orientable sequences with period close to the
maximum.

As a result, in the remainder of this paper we consider ways of constructing sets of edges in
$B_{k}(n-1)$ that define an antisymmetric Eulerian subgraph of $B_{k}(n-1)$.

\section{A simple construction}  \label{section:construction}

\subsection{The subgraph}

We next give a very simple way of establishing an antisymmetric Eulerian subgraph of $B_{k}(n-1)$.

\begin{definition}  \label{definition:A_n(k)}
If $n\geq2$ and $k\geq3$ let $A_k(n)$ be the subgraph of $B_{k}(n-1)$ with edges corresponding to
the $n$-tuples $(a_0,a_1,\ldots,a_{n-1})$, where $a_i\in\mathbb{Z}_k$, for which
$a_{n-1}-a_0\in\{1,2,\ldots\lfloor(k-1)/2\rfloor\}$.
\end{definition}

It is simple to establish how many edges there are in $A_k(n)$.

\begin{lemma}  \label{lemma:Akn}
Suppose $n\geq2$ and $k\geq3$.  Then if $E$ is the set of edges in $A_k(n)$
\[ |E|= k^{n-1}\left\lfloor\frac{k-1}{2}\right\rfloor.\]
\end{lemma}

\begin{proof}
Suppose $(a_0,a_1,\ldots,a_{n-1})$ is an edge in $A_k(n)$.  Then there are $k$ choices for each
$a_i$, $0\leq i\leq n-2$.  Finally, there are $\lfloor(k-1)/2\rfloor$ choices for $a_{n-1}$ and
the
result follows.
\end{proof}

The following two results demonstrate the importance of $A_k(n)$.

\begin{lemma}  \label{lemma:tuples}
Suppose $n\geq2$ and $k\geq3$.  $A_k(n)$ is antisymmetric in the de Bruijn digraph $B_{k}(n-1)$.
\end{lemma}

\begin{proof}
If $\mathbf{s}=(s_0,s_1,\ldots,s_{n-1})$ is an edge in $A_k(n)$, then by definition
$s_{n-1}-s_0\in\{1,2,\ldots\lfloor(k-1)/2\rfloor\}$.  Hence
$s_0-s_{n-1}\in\{n-1,n-2,\ldots\lfloor(k+2)/2\rfloor\}$, so that $\mathbf{s}^R=(s_{n-1},s_{n-2},\ldots,s_0)$ is not an edge in $A_k(n)$, since
\[ \{1,2,\ldots\lfloor(k-1)/2\rfloor\}\cap\{k-1,k-2,\ldots\lfloor(k+2)/2\rfloor\}=\emptyset \]
and the result follows.
\end{proof}

We can now give the following.

\begin{theorem}  \label{theorem:A_n(k)_Eulerian}
Suppose $n\geq2$ and $k\geq5$.  $A_k(n)$ is an antisymmetric Eulerian subgraph of $B_{k}(n-1)$
with
$k^{n-1}\left\lfloor\frac{k-1}{2}\right\rfloor$ edges.
\end{theorem}

\begin{proof}
From Lemmas~\ref{lemma:tuples} and \ref{lemma:Akn} we need only show that $A_k(n)$ is connected
and
that every vertex has in-degree equal to its out-degree.

Consider any vertex $\mathbf{u}=(u_1,u_2,\ldots,u_{n-1})$ in $A_k(n)$. An incoming edge
\[ (s,u_1,u_2,\ldots,u_{n-1}) \]
must satisfy $u_{n-1}-s\in\{1,2,\ldots\lfloor(k-1)/2\rfloor\}$. Similarly an outgoing edge
\[(u_1,u_2,\ldots,u_{n-1},t)\]
must satisfy $t-u_1\in\{1,2,\ldots\lfloor(k-1)/2\rfloor\}$.  Regardless of the values of $u_1$ and
$u_{n-1}$ there are clearly $\lfloor(k-1)/2\rfloor$ possible values for both $s$ and $t$, i.e.\
the
in-degree of every vertex is the same as its out-degree; in both cases it will equal
$\lfloor(k-1)/2\rfloor$.

To show that $A_k(n)$ is connected observe that there will always exist a path from
\[(u_1,u_2,\ldots,u_{n-1}) \text{~to~} (u_1,u_2,\ldots,u_{j-1},u_{j}+1,u_{j+1},\ldots,u_{n-1}) \]
for any $j$.  This follows since, if $k\geq 5$, there is an edge from
\[ (u_1,u_2,\ldots,u_{n-1}) \text{~to~} (u_2,u_3,\ldots,u_{n-1},u_1+1) \]
and also an edge from
\[ (u_1,u_2,\ldots,u_{n-1}) \text{~to~} (u_2,u_3,\ldots,u_{n-1},u_1+2). \]
This enables a path of length $n$ to be constructed from $(u_1,u_2,\ldots,u_{n-1})$ to
$(u_1+1,u_2+1,\ldots,u_{j-1}+1,u_{j}+2,u_{j+1}+1,\ldots,u_{n-1}+1)$ for any $j$, and then a path
of
length $2n$ to $(u_1+2,u_2+2,\ldots,u_{j-1}+2,u_{j}+3,u_{j+1}+2,\ldots,u_{n-1}+2)$, etc., to get a
path to $(u_1+k,u_2+k,\ldots,u_{j-1}+k,u_{j}+k+1,u_{j+1}+k,\ldots,u_{n-1}+k)$ --- which is equal
to
$(u_1,u_2,\ldots,u_{j-1},u_{j}+1,u_{j+1},\ldots,u_{n-1})$.

Hence, inductively, there exists a path from $(u_1,u_2,\ldots,u_{n-1})$ to any vertex
$(v_1,v_2,\ldots,v_{n-1})$.

The result follows.
\end{proof}

\begin{remark}
We need to assume $k\geq5$ since, unfortunately, $A_k(n)$ is not connected if $k<5$.  If $k=3$ and
$n=3$ we have:
\[ 00\rightarrow01\rightarrow11\rightarrow12\rightarrow22\rightarrow20\rightarrow00 \]
and
\[ 02\rightarrow21\rightarrow10\rightarrow02 \]
i.e.\ circuits of length 6 and 3. Similarly if $k=3$ and $n=4$ we obtain three circuits each of
length 9.  Similar problems arise for $k=4$, because for $k=3$ and $k=4$ the in-degree and
out-degree of every vertex in $A_n(k)$ is only 1.
\end{remark}

Combining Theorems~\ref{theorem:A_n(k)_Eulerian} and \ref{theorem:correspondence} we immediately
obtain the following.

\begin{corollary}  \label{corollary:OSkn}
If $n\geq2$ and $k\geq5$ then there exists an $\mathcal{OS}_k(n)$ with period
\[k^{n-1}\left\lfloor\frac{k-1}{2}\right\rfloor.\]
\end{corollary}

\begin{example}
As an example of Corollary~\ref{corollary:OSkn}, consider the case $k=5$ and $n=3$.  The 50
3-tuples in $A_5(3)$ are listed in Table~\ref{table:5-ary_3_tuples}, and a period of an
$\mathcal{OS}_5(3)$ containing these 50 3-tuples is:
\[ [00123~40112~23344~00213~24304~21431~03142~03204~10224~41133]. \]

\begin{table}[htb]
\caption{$5$-ary $3$-tuples in $A_5(3)$} \label{table:5-ary_3_tuples}
\begin{center}

\begin{tabular}{|cc|cc|cc|cc|cc|}
\hline
001 & 002 & 102 & 103 & 203 & 204 & 304 & 300 & 400 & 401  \\
011 & 012 & 112 & 113 & 213 & 214 & 314 & 310 & 410 & 411  \\
021 & 022 & 122 & 123 & 223 & 224 & 324 & 320 & 420 & 421  \\
031 & 032 & 132 & 133 & 233 & 234 & 334 & 330 & 430 & 431  \\
041 & 042 & 142 & 143 & 243 & 244 & 344 & 340 & 440 & 441  \\
\hline
\end{tabular}
\end{center}
\end{table}
\end{example}

Table~\ref{table:OS_periods} tabulates the periods of the generated sequences for small $k$ and
$n$.   In each case the upper bound for the period (from Theorem~\ref{theorem:OS_bound_4}) is given
in brackets.

\begin{table}[htb]
\caption{$\mathcal{OS}_q(n)$ periods (and bounds)} \label{table:OS_periods}
\begin{center}

\begin{tabular}{c|rrrrr} \hline
$n$ & $k=5$    & $k=6$    & $k=7$     & $k=8$     & $k=9$      \\ \hline
2   & 10       & 12       & 21        & 24        & 36         \\
    & (10)     & (12)     & (21)      & (24)      & (36)       \\ \hline
3   & 50       & 72       & 147       & 192       & 324        \\
    & (50)     & (84)     & (147)     & (216)     & (324)      \\ \hline
4   & 250      & 432      & 1029      & 1536      & 2916       \\
    & (280)    & (612)    & (1134)    & (1984)    & (3168)     \\ \hline
5   & 1250     & 2592     & 7203      & 12288     & 26244      \\
    & (1450)   & (3684)   & (8085)    & (15896)   & (28836)    \\ \hline
6   & 6250     & 15552    & 50421     & 98304     & 236196     \\
    & (7550)   & (23019)  & (58065)   & (130332)  & (264060)   \\ \hline
7   & 31250    & 93312    & 352947    & 786432    & 2125764    \\
    & (38100)  & (138144) & (408072)  & (1042712) & (2381400) \\ \hline
8   & 156250   & 559872   & 2470629   & 6291456   & 19131876   \\
    & (193800) & (837879) & (2876496) & (8382492) & (21507120) \\ \hline
\end{tabular}

\end{center}
\end{table}

Comparison of the bound in Theorem~\ref{theorem:OS_bound_4} with Corollary~\ref{corollary:OSkn}
shows that the sequences have optimal period for $n=2$ and $n=3$ if $k$ is odd.

\subsection{Variants}

There are a number of other antisymmetric Eulerian subgraphs of $B_{k}(n-1)$ with the same number
of edges as $A_k(n)$.  We briefly mention two such examples, where in both cases we assume $k\geq
5$ and $n\geq2$.
\begin{itemize}
\item For odd $k$, define $C_k(n)$, where $(a_0,a_1,\ldots,a_{n-1})$ is an edge in $C_k(n)$ if
    and only if $a_{n-1}-a_0$ is odd. It is immediate to see that $C_k(n)$ is antisymmetric
    since if $(a_0,a_1,\ldots,a_{n-1})$ is an edge then $a_{0}-a_{n-1}$ is even, since $k$ is
    odd.  It can be shown that $C_k(n)$ is Eulerian using a similar argument to that for
    $A_n(k)$.

\item A second variant constrains more than the first and last elements of an $n$-tuple.
    Suppose $t\leq n/2$. Let $A_k(n,t)$ be the subgraph of $B_{k}(n-1)$ with edges equal to
the
    following set of $k$-ary $n$-tuples
\[ \left\{ (a_0,a_1,\ldots,a_{n-1}) :  \sum_{i=n-t}^{n-1}a_i-\sum_{i=0}^{t-1} a_i
\in\{1,2,\ldots\lfloor(k-1)/2\rfloor\} \right\}.  \]
If the $n$-tuple $\mathbf{s}$ is an element of $A_k(n,t)$, then $\mathbf{s}^R\not\in
A_k(n,t)$,
i.e.\ it is antisymmetric.  It again follows that it is Eulerian using a similar argument to
that for $A_k(n)$.
\end{itemize}
There are, no doubt, further variants that could be devised.

\section{A construction using the Lempel homomorphism}  \label{section:Lempel_construction}

We next show how, using the inverse of the Lempel homomorphism, a subgraph of $B_{k}(n-1)$ with
certain special properties can be used to construct an antisymmetric Eulerian subgraph of
$B_{k}(n)$.

\subsection{Preliminaries}

We first need to define the Lempel homomorphism $D$, that maps from $B_k(n)$ to $B_{k}(n-1)$. We
follow the definition of Alhakim and Akinwande \cite{Alhakim11}, who generalised the original
Lempel definition \cite{Lempel70}, that only applied for $k=2$, to alphabets of arbitrary size.

\begin{definition} \label{Lempel}
Define the function $D_{\beta}:B_k(n)\rightarrow B_{k}(n-1)$ as follows, where
$\beta\in\mathbb{Z}^*_k$. If $\mathbf{a} = (a_0,a_1,\ldots,a_{n-1})\in \mathbb{Z}_k^n$ then
\[ D_{\beta}(\mathbf{a}) = (\beta(a_1-a_0),\beta(a_2-a_1),\ldots\beta(a_{n-1}-a_{n-2})). \]
\end{definition}

Clearly $D_{\beta}$ is onto if and only if $\text{gcd}(\beta,k) = 1$.  We also have the following.

\begin{lemma}[Alhakim and Akinwande, \cite{Alhakim11}]
Every vertex in $B_{k}(n-1)$ has exactly $k$ images in $B_k(n)$ under $D_{\beta}^{-1}$ if and only
if $\beta$ is coprime to $k$.
\end{lemma}

In the remainder of this paper we are particularly interested in the case $\beta=1$, and we simply
write $D$ for $D_1$.

\subsection{Using the inverse homomorphism}

The following definition is key to the construction.

\begin{definition}
Suppose $T$ is a subgraph of the de Bruijn digraph $B_{k}(n-1)$ for some $n\geq2$ and $k\geq2$.
$T$ is said to be \emph{antinegasymmetric} if the following property holds.

Suppose $\mathbf{x}=(x_0,x_1,\ldots,x_{n-2}),\mathbf{y}=(y_0,y_1,\ldots,y_{n-2})\in
\mathbb{Z}_k^{n-1}$, i.e.\ they are vertices in $T$.  Then if $(\mathbf{x},\mathbf{y})$ is an
edge in $T$, then $(-\mathbf{y}^R,-\mathbf{x}^R)$ is not an edge in $T$.
\end{definition}

\begin{lemma}  \label{lemma:Lempel_antisymmetry}
Suppose $n\geq2$ and $k\geq3$.  If $T$ is an antinegasymmetric subgraph of the de Bruijn
digraph $B_{k}(n-1)$ with edge set $E$, then $D^{-1}(E)$, of cardinality $k|E|$, is the set of
edges for an antisymmetric subgraph of $B_{k}(n)$, which, abusing our notation slightly, we
refer to as $D^{-1}(T)$. Moreover, if every vertex of $T$ has in-degree equal to its
out-degree, then the same apples to $D^{-1}(T)$.
\end{lemma}

\begin{proof}
Suppose $D^{-1}(T)$ is not antisymmetric, i.e.\ there exist $(n+1)$-tuples
$\mathbf{a}=(a_0,a_1,\ldots,a_{n}),\mathbf{b}=(b_0,b_1,\ldots,b_{n})\in D^{-1}(E)$ such that
$\mathbf{a}=\mathbf{b}^R$, i.e.\ $a_i=b_{n-i}$, for $0\leq i\leq n$.

Suppose also that $D(\mathbf{a})=\mathbf{c}$ and $D(\mathbf{b})=\mathbf{d}$, where
$\mathbf{c}=(c_0,c_1,\ldots,c_{n-1}),\mathbf{d}=(d_0,d_1,\ldots,d_{n-1})\in E$. Hence
$c_i=a_{i+1}-a_i$ and $d_i=b_{i+1}-b_i$ for $0\leq i\leq n-1$.  Since $a_i=b_{n-i}$ for $0\leq
i\leq n$, for any $j$ ($0\leq j\leq n-1$) we have:
\[ c_j = a_{j+1}-a_j = b_{n-(j+1)}-b_{n-j} = -(b_{n-j}-b_{(n-1)-j})= -d_{(n-1)-j}. \]
Hence $\mathbf{c}=-\mathbf{d}^R$, but this contradicts the assumption that $T$ is
antinegasymmetric.

Every edge in $E$ corresponds to $k$ edges in $D^{-1}(E)$, and hence $|D^{-1}(E)|=k|E|$.

It remains to show that every vertex of $D^{-1}(T)$ has in-degree equal to its out-degree.
Suppose $\mathbf{x}=(x_1,x_2,\ldots,x_{n})$ is a vertex of $D^{-1}(T)$.  For every edge
$\mathbf{a}\in D^{-1}(T)$ that ends in $\mathbf{x}$, there is a corresponding edge
$D(\mathbf{a})\in T$ that ends in $D(\mathbf{x})$.  That is, the in-degree of $\mathbf{x}$ in
$D^{-1}(T)$ will equal the in-degree of $D(\mathbf{x})$ in $T$.  An exactly similar result
holds for out-degree.  Since every vertex of $T$ has in-degree equal to its out-degree, the
same holds for $D^{-1}(T)$.
\end{proof}

\begin{example}
Observe that $E=\{(1,0),(1,1),(1,2),(0,1),(2,1)\}$ is the set of edges for an antinegasymmetric
subgraph of the de Bruijn digraph $B_{4}(1)$. Observe that the edges of $T$ are all the $4$-ary
$2$-tuples containing a $1$ and not containing a $3=-1$; this clearly guarantees antinegasymmetry.
It is also simple to see that the in-degree of every vertex is the same as its out-degree.

Then $D^{-1}(E)$ is equal to
\[ \{(a,a+1,a+1),(a,a+1,a+2),(a,a+1,a+3),(a,a,a+1),(a,a+2,a+3):a\in\mathbb{Z}_4\} \]
which is a set of 20 edges, which by Lemma~\ref{lemma:Lempel_antisymmetry} forms an antisymmetric
subgraph of the de Bruijn digraph $B_{4}(2)$ with in-degree equal to out-degree for every vertex.
It is also simple to check that the graph is connected (ignoring vertices with in-degree zero),
and
hence is Eulerian. As a result Eulerian circuits exist, each of which corresponds to an
$\mathcal{OS}_4(3)$ of period $20$. One such orientable sequence is the $\mathcal{OS}_4(3)$
$[00112012230130231233]$ of period $20$ found by Gabri\'{c} and Sawada \cite{Gabric25} via an
exhaustive search.
\end{example}

\subsection{Constructing antinegasymmetric subgraphs}

We next demonstrate how to construct antinegasymmetric subgraphs of the de Bruijn digraph
$B_{k}(n-1)$ for every $n\geq2$.  This builds on the work described in \cite{Mitchell25a}.

\begin{definition}[\cite{Mitchell25a}]  \label{definition:pseudoweight}
Suppose $\mathbf{u}=(u_0,u_1,\ldots,u_{n-1})$ is an $n$-tuple of elements of $\mathbb{Z}_k$
($k>1$). Define the function $f:\mathbb{Z}_k\rightarrow\mathbb{Q}$ as follows: for any
$u\in\mathbb{Z}_k$ treat $u$ as an integer in the range $[0,k-1]$ and set $f(u)=u$ if $u\neq0$ and
$f(u)=k/2$ if $u=0$. Then the \emph{pseudoweight} of $\mathbf{u}$ is defined to be the sum
\[ w^*(\mathbf{u}) = \sum_{i=0}^{n-1}f(u_i) \]
where the sum is computed in $\mathbb{Q}$.
\end{definition}

As a simple example for $k=3$, the 4-tuple $(0,1,1,2)$ has pseudoweight $1.5+1+1+2=5.5$, since
$f(0)=\frac{3}{2}$.

The following result is closely related to \cite[Theorem 3.14]{Mitchell25a}.

\begin{theorem}  \label{theorem:pseudoweight}
Suppose $n\geq2$ and $k\geq3$. If $E$ is the set of all $k$-ary $n$-tuples with pseudoweight less
than $nk/2$, then $E$ is the set of edges for an antinegasymmetric subgraph of the de Bruijn
digraph $B_{k}(n-1)$. Moreover, every vertex in this subgraph has in-degree equal to its
out-degree.
\end{theorem}

\begin{proof}
Consider any $n$-tuple $\mathbf{u}=(u_0,u_1,\ldots,u_{n-1})\in E$.  By definition we know that
$w^*(\mathbf{u})<nq/2$. We claim that $w^*(-\mathbf{u}^R)=nk-w^*(\mathbf{u})$. This follows
immediately from the definition of $f$ since $f(-u_i)=k-f(u_i)$ for every possible value of $u_i$.

Hence, since $w^*(\mathbf{u})<nk/2$, it follows immediately that $w^*(-\mathbf{u}^R)>nk/2$.  Thus
the $n$-tuples in $E$ are all distinct from the $n$-tuples in
$-E^R=\{-\mathbf{u}^R~:~\mathbf{u}\in
E\}$. Hence $E$ is the set of edges for an antinegasymmetric subgraph of the de Bruijn digraph
$B_{k}(n-1)$.

It remains to show that the in-degree of every vertex is equal to its out-degree. Consider any
vertex $\mathbf{u}=(u_0,u_1,\ldots,u_{n-2})$ of $B_{k}(n-1)$. An incoming edge
\[ (s,u_0,u_1,\ldots,u_{n-2}) \] in $E$
must satisfy $s+w^*(\mathbf{u})<nk/2$. Similarly an outgoing edge
\[(u_0,u_1,\ldots,u_{n-2},t)\]
in $E$ must satisfy $t+w^*(\mathbf{u})<nk/2$.  That is, the in-degree of every vertex is the same
as its out-degree.
\end{proof}

It is clearly of interest to know the cardinality of the edge set $E$ of
Theorem~\ref{theorem:pseudoweight}.  Again following \cite{Mitchell25a} we make the following
definition.

\begin{definition}
If $k\geq2$ and $n\geq 1$, let $r_{k,n,s}$ denote the number of $k$-ary $n$-tuples with
pseudoweight exactly $s$, where $r_{k,n,s}=0$ by definition if $s<n$ or $s>n(k-1)$.
\end{definition}

We immediately have the following corollary of Theorem~\ref{theorem:pseudoweight}.

\begin{corollary}  \label{corollary:E_pseudoweight}
If $n\geq2$ and $k\geq3$, then $E$ as defined in Theorem~\ref{theorem:pseudoweight} is the set of
edges for an antinegasymmetric subgraph of the de Bruijn digraph $B_{k}(n-1)$ containing
$\frac{k^n-r_{k,n,nk/2}}{2}$ edges and where every vertex in this subgraph has in-degree equal to
its out-degree.
\end{corollary}

A discussion of the properties of $r_{k,n,s}$ and a table of small values is given in \cite[Section
3.4]{Mitchell25a}.

From Lemma~\ref{lemma:Lempel_antisymmetry}, this means that $D^{-1}(E)$, where $E$ is as defined
in
Theorem~\ref{theorem:pseudoweight}, is the set of edges for an antisymmetric subgraph of $B_k(n)$,
where every vertex has in-degree equal to its out-degree.  It would be ideal if we could also show
that $D^{-1}(E)$ is the set of edges for a \emph{connected} subgraph of $B_{k}(n)$ and then we
would know that $D^{-1}(E)$ is the set of edges for an Eulerian antisymmetric subgraph of
$B_k(n)$.
We could then apply Theorem~\ref{theorem:correspondence} to obtain an $\mathcal{OS}_k(n+1)$ of
period $|D^{-1}(E)|=k|E|$.

Thus, showing that the subgraph is connected, after removal of any vertices with in-degree and
out-degree zero, is key.  This is our next focus.

\subsection{Establishing connectivity}

We next establish the key result that if $E$ is as in Theorem~\ref{theorem:pseudoweight} then
$D^{-1}(E)$ is the set of edges for a \emph{connected} subgraph of $B_{k}(n)$.  We first need the
following simple lemma.

\begin{lemma}  \label{lemma:degree-zero}
Suppose $n\geq2$ and $k\geq 3$. Suppose $E$ is as in Theorem~\ref{theorem:pseudoweight}. If
$\mathbf{a}=(a_0,a_1,\ldots,a_{n-1})$ is a vertex in the subgraph of $B_k(n)$ for which
$D^{-1}(E)$
is the set of edges, then $\mathbf{a}$ has in-degree and out-degree zero if and only if
\[ w^*(D(\mathbf{a}))=w^*((a_{1}-a_0,a_{2}-a_1,\ldots,a_{n-1}-a_{n-2}))\geq \frac{nk}{2}-1. \]
\end{lemma}

\begin{proof}
Suppose $(x,a_0,a_1,\ldots,a_{n-1})$ is an incoming edge to the vertex $\mathbf{a}$, for some $x$.
Then $D((x,a_0,a_1,\ldots,a_{n-1}))$ must be an edge in the subgraph of $B_k(n-1)$ for which $E$
is
the set of edges. Now
$D((x,a_0,a_1,\ldots,a_{n-1}))=(a_0-x,a_{1}-a_0,a_{2}-a_1,\ldots,a_{n-1}-a_{n-2})$ which must be
an
edge in $E$.  Thus, by definition of $E$:
\[ w^*((a_0-x,a_{1}-a_0,a_{2}-a_1,\ldots,a_{n-1}-a_{n-2}))< nk/2 \]
It is straightforward to see that
\begin{align*}
w^*((a_0-x,a_{1}-a_0,a_{2}-a_1,\ldots,a_{n-1}-a_{n-2})) =\\
w^*((a_0-x))+w^*((a_{1}-a_0,a_{2}-a_1,\ldots,a_{n-1}-a_{n-2})).
\end{align*}
and $w^*((a_0-x))\geq 1$; hence
\[ w^*((a_{1}-a_0,a_{2}-a_1,\ldots,a_{n-1}-a_{n-2}))< nk/2-1. \]
That is, $\mathbf{a}$ has in-degree and out-degree zero if
\[ w^*((a_{1}-a_0,a_{2}-a_1,\ldots,a_{n-1}-a_{n-2}))\geq nk/2-1. \]

To show that $\mathbf{a}$ has non-zero in-degree and out-degree if
\[ w^*((a_{1}-a_0,a_{2}-a_1,\ldots,a_{n-1}-a_{n-2}))< nk/2-1 \]
it suffices to point out that the edge $D((a_0-1,a_0,a_1,\ldots,a_{n-1}))$ will have pseudoweight
less than $1+(nk/2-1)=nk/2$, and hence $(a_0-1,a_0,a_1,\ldots,a_{n-1})$ is an edge in $E$.  The
result follows.
\end{proof}

We can now establish the main result.

\begin{theorem}  \label{theorem:D-1E_connected}
Suppose $n\geq2$ and $k\geq 3$. If $E$ is defined as in Theorem~\ref{theorem:pseudoweight} then
$D^{-1}(E)$ is the set of edges of a connected subgraph of $B_{k}(n)$.
\end{theorem}

\begin{proof}
By definition, the edges of $E$ consist of $k$-ary $n$-tuples with pseudoweight less than $kn/2$.
Suppose $\mathbf{a}=(a_0,a_1,\ldots,a_{n-1})$ and $\mathbf{b}=(b_0,b_1,\ldots,b_{n-1})$ are
vertices in the subgraph of $B_k(n)$ for which $D^{-1}(E)$ is the set of edges, satisfying
$w^*(D(\mathbf{a}))<nk/2-1$ and $w^*(D(\mathbf{b}))<nk/2-1$. We need to show that there is a path
from $\mathbf{a}$ to $\mathbf{b}$ in this subgraph.

The proof relies on three simple observations.
\begin{itemize}
\item {\bf Observation A}.  Suppose $\mathbf{c}=(c_0,c_1,\ldots,c_{n-1})$ is any vertex in the
    subgraph with non-zero in-degree, i.e.\ $w^*(D(\mathbf{c}))<nk/2-1$.  Consider the
    $(n+1)$-tuple $\mathbf{c}^+=(c_0,c_1,\ldots,c_{n-1},c_{n-1}+1)$. Now, by definition,
    $\mathbf{c}^+$ is an element of $D^{-1}(E)$ if and only if $w^*(D(\mathbf{c}^+))<nk/2$.
    However, trivially,
\[ w^*(D(\mathbf{c}^+))=w^*(D(\mathbf{c}))+1< (nk/2-1)+1=nk/2. \]
Hence $\mathbf{c}^+$ is always an element of $D^{-1}(E)$, and so there is always an edge from
$\mathbf{c}=(c_0,c_1,\ldots,c_{n-1})$ to $(c_1,c_2,\ldots,c_{n-1},c_{n-1}+1)$, as long as
$\mathbf{c}$ has non-zero out-degree.
\item {\bf Observation B}.  Suppose $d$ is an arbitrary element of $Z_k$, and consider the
    $(n+1)$-tuple $\mathbf{d}=(d,d+1,\ldots,d+n-1,d+n)$.  It is simple to verify that
    \[ w^*(D(\mathbf{d}))=n. \]
Now, since $k\geq3$, $w^*(D(\mathbf{d}))<nk/2$, and hence $\mathbf{d}$ is an element of
$D^{-1}(E)$.  Thus there is always an edge from $(d,d+1,\ldots,d+n-1)$ to
$(d+1,d+2,\ldots,d+n)$ for any $d$, and hence there is always a directed path from
$(d,d+1,\ldots,d+n-1)$ to $(e,e+1,\ldots,e+n-1)$ for any $d$ and $e$.
\item {\bf Observation C}.  Suppose $\mathbf{c}=(c_0,c_1,\ldots,c_{n-1})$ is any vertex in the
    subgraph with non-zero in-degree, i.e.\ $w^*(D(\mathbf{c}))<nk/2-1$.  Consider the
    $(n+1)$-tuple $\mathbf{c}^-=(c_0-1,c_0,c_1,\ldots,c_{n-1})$.  As previously
    \[ w^*(D(\mathbf{c}^-))=w^*(D(\mathbf{c}))+1< (nk/2-1)+1=nk/2. \]
Hence there is always an edge from $(c_0-1,c_0,c_1,\ldots,c_{n-2})$ to
$(\mathbf{c}=(c_0,c_1,\ldots,c_{n-1})$, as long as $\mathbf{c}$ has non-zero in-degree.
\end{itemize}

The proof now follows in three stages.
\begin{itemize}
\item Applying Observation A $n-1$ times, there exists a directed path from $\mathbf{a}$ to
the
    vertex $(a_{n-1},a_{n-1}+1,\ldots,a_{n-1}+(n-1))$.
\item From Observation B, there exists a directed path from
    $(a_{n-1},a_{n-1}+1,\ldots,a_{n-1}+n-1)$ to $(b_0-n+1,b_0-n+2,\ldots,b_0)$.
\item Applying Observation C $n-1$ times, there exists a directed path from
    $(b_0-n+1,b_0-n+2,\ldots,b_0)$ to $\mathbf{b}=(b_0,b_1,\ldots,b_{n-1})$.
\end{itemize}
That is, there exists a directed path from $\mathbf{a}$ to $\mathbf{b}$, and the result follows.
\end{proof}

Combining Corollary~\ref{corollary:E_pseudoweight} with Theorems~\ref{theorem:correspondence} and
\ref{theorem:D-1E_connected} gives the following.

\begin{corollary}  \label{corollary:OS_from_antineg}
If $k\geq3$ and $n\geq 2$ there exists an $\mathcal{OS}_k(n+1)$ with period
\[ k\frac{k^n-r_{k,n,nk/2}}{2}. \]
\end{corollary}

\begin{remark}
We have defined a set of edges forming an Eulerian subgraph of $B_k(n)$, and every Eulerian
circuit
in this subgraph will yield an $\mathcal{OS}_k(n+1)$.  This approach will thus yield many
different
such sequences, since there will be many possible Eulerian circuits.
\end{remark}

\begin{example}
As an example of Corollary~\ref{corollary:OS_from_antineg}, consider the case $k=3$ and $n=3$.
The
ten 3-ary 3-tuples having pseudoweight less than 4.5 are listed in
Table~\ref{table:3-ary_3_tuples}
--- these form the set $E$.  The set $D^{-1}(E)$ consists of the 30 4-tuples given in
Table~\ref{table:3-ary_4_tuples}, where the 4-tuples are grouped in threes according to the
element
of $E$ of which they are pre-images under $D$.  Finally, a period of an $\mathcal{OS}_5(3)$
containing the 4-tuples in $D^{-1}(E)$ is:
\[ [01201~21202~01012~22011~20011~12200] \].

\begin{table}[htb]
\caption{$E$: $3$-ary $3$-tuples with pseudoweight less than $4.5$} \label{table:3-ary_3_tuples}
\begin{center}

\begin{tabular}{ccc}
\hline
111 \\
011 & 101 & 110 \\
001 & 010 & 100 \\
112 & 121 & 211 \\
\hline
\end{tabular}

\end{center}
\end{table}

\begin{table}[htb]
\caption{$D^{-1}(E)$: $3$-ary $4$-tuples} \label{table:3-ary_4_tuples}
\begin{center}

\begin{tabular}{ccc|ccc|ccc}
\hline
0120 & 1201 & 2012 &   \\
0012 & 1120 & 2201 & 0112 & 1220 & 2001 & 0122 & 1200 & 2011 \\
0001 & 1112 & 2220 & 0011 & 1122 & 2200 & 0111 & 1222 & 2000 \\
0121 & 1202 & 2010 & 0101 & 1212 & 2020 & 0201 & 1012 & 2120 \\
\hline
\end{tabular}

\end{center}
\end{table}

\end{example}

The periods of the sequences of Corollary~\ref{corollary:OS_from_antineg} for small $k$ and $n$ are
given in Table~\ref{table:OS_antineg_periods_bounds}, along with the bound on the period from
Theorem~\ref{theorem:OS_bound_4}. In the cases $n=3$ and $n=4$ ($k$ odd), the periods of the
sequences of Corollary~\ref{corollary:OS_from_antineg} meet the bound of
Theorem~\ref{theorem:OS_bound_4}.

\begin{table}[htb]
\centering \caption{Periods of the constructed $\mathcal{OS}_k(n)$ (and bounds)}
\label{table:OS_antineg_periods_bounds}

\begin{tabular}{c|rrrrrr} \hline
$n$ & $k=3$   & $k=4$    & $k=5$    & $k=6$    & $k=7$     & $k=8$     \\ \hline
3   & 9       & 20       & 50       & 84       & 147       & 216       \\
    & (9)     & (20)     & (50)     & (84)     & (147)     & (216)     \\ \hline
4   & 30      & 88       & 280      & 534      & 1134      & 1800      \\
    & (30)    & (112)    & (280)    & (612)    & (1134)    & (1984)    \\ \hline
5   & 93      & 372      & 1390     & 3300     & 7763      & 14680     \\
    & (999)   & (452)    & (1450)   & (3684)   & (8085)    & (15896)   \\ \hline
6   & 288     & 1544     & 7160     & 20172    & 56056     & 118864    \\
    & (315)   & (1958)   & (7550)   & (23019)  & (58065)   & (130332)  \\ \hline
7   & 882     & 6344     & 35810    & 122646   & 388626    & 959160    \\
    & (972)   & (7844)   & (38100)  & (138144) & (408072)  & (1042712) \\ \hline
8   & 2691    & 25904    & 181100   & 743370   & 2757937   & 7724552   \\
    & (3096)  & (32390)  & (193800) & (837879) & (2876496) & (8382492) \\ \hline
\end{tabular}

\end{table}

\section{Relation to other work}  \label{section:other_work}

It is of interest to consider how the construction methods described here affect efforts to find
orientable sequences with the largest possible period.  The current state of knowledge for small
$n$ and $k>2$ in this direction is summarised in Table~\ref{table:OS_periods_bounds}, where the
upper bound from Theorem~\ref{theorem:OS_bound_4} is given in brackets beneath the largest known
period.

\begin{table}[htb]
\centering
\caption{Largest known periods for an $\mathcal{OS}_k(n)$ (and bounds)}
\label{table:OS_periods_bounds}

\begin{tabular}{crrrrrr} \hline
$n$ & $k=3$   & $k=4$    & $k=5$    & $k=6$    & $k=7$     & $k=8$     \\ \hline
2   & {\bf 3} & {\bf 4}  & {\bf 10} & {\bf 12} & {\bf 21}  & {\bf 24}  \\
    & (3)     & (4)      & (10)     & (12)     & (21)      & (24)      \\ \hline
3   & {\bf 9} & {\bf 20} & {\bf 50} & {\bf 84} & {\bf 147} & {\bf  216}\\
    & (9)     & (20)     & (50)     & (84)     & (147)     & (216)     \\ \hline
4   & {\bf 30}& 88       & {\bf 280}& 534      & {\bf 1134}&  1800     \\
    & (30)    & (112)    & (280)    & (612)    & (1134)    & (1984)    \\ \hline
5   & 93      & 372      & 1390     & 3360     & 7763      &  15120    \\
    & (99)    & (452)    & (1450)   & (3684)   & (8085)    & (15896)   \\ \hline
6   & 288     & 1608     & 7160     & 21150    & 56056     &  124320   \\
    & (315)   & (1958)   & (7550)   & (23019)  & (58065)   & (130332)  \\ \hline
7   & 882     & 7308     & 36890    & 135450   & 403389    &  1034264  \\
    & (972)   & (7844)   & (38100)  & (138144) & (408072)  & (1042712) \\ \hline
8   & 2691    & 30300    & 187980   & 821940   & 2844408   &  8315496  \\
    & (3096)  & (32390)  & (193800) & (837879) & (2876496) & (8382492) \\ \hline
\end{tabular}
\end{table}

The following observations can be made about this table.  The bound values (in brackets) follow
from Theorem~\ref{theorem:OS_bound_4}. The values in bold represent maximal values.
\begin{itemize}
\item $n=2$:  The fact that there exists an $\mathcal{OS}_k(2)$ with period meeting the bound
    follows from \cite[Theorem 5.4]{Alhakim24a} (for $k$ prime), and from \cite[Theorem
    2]{Gabric25} and \cite[Lemma 2.2]{Mitchell25a} for general $k$.
\item $n=3$: The existence of an $\mathcal{OS}_k(3)$ meeting the period bound is due to
    \cite[Example 5.1]{Alhakim24a} for $k=3$, \cite[Section 1]{Gabric25} for $k=5$ (from an
    exhaustive search), and for general $k$ from this paper.
\item $n=4$: The fact that the maximum period of an $\mathcal{OS}_3(4)$ is 30 is again due to
    \cite[Section 1]{Gabric25}, and the existence of an $\mathcal{OS}_k(4)$ meeting the bound
    for general odd $k$ is from this paper. All values for $4\leq k\leq 8$ follow from this
    paper.
\item $n=5$: The $\mathcal{OS}_3(5)$, $\mathcal{OS}_4(5)$, $\mathcal{OS}_5(5)$ and
    $\mathcal{OS}_7(5)$, respectively of periods 93, 372, 1390 and 7763, come from this paper.
    The values for $k=6$ and $k=8$ come from \cite[Theorem 11]{Gabric25} from cycle-joining.
\item $n\geq6$: The $\mathcal{OS}_3(6)$, $\mathcal{OS}_5(6)$, $\mathcal{OS}_3(7)$ and
    $\mathcal{OS}_3(8)$, of periods 288, 7160, 882 and 2691 respectively, come from this paper.
    All other values in the table come from \cite[Theorem 11]{Gabric25}.
\end{itemize}

\section{Conclusions}  \label{section:conclusions}

We have presented tighter upper bounds on the period of an orientable sequence and a general
approach for constructing orientable sequences for any $n\geq2$ and any $k\geq3$. For $n\leq3$, and
$n=4$ when $k$ is odd, the results show that the bound on the period is tight. The construction
method gives sequences of slightly lesser period than the approach of Gabri\'{c} and Sawada
\cite[Theorem 11]{Gabric25} for larger $n$, although all the sequences constructed here have period
greater than $(k-1)/k$ of the maximum for odd $k$, and greater than $(k-2)/k$ of the maximum for
even $k$.

Table~\ref{table:OS_periods_bounds} reveals that further work is required to close the gap between
the known largest period and the best existing upper bound for $k$ even and $n=4$ as well as for
all $k$ when $n>4$.


\providecommand{\bysame}{\leavevmode\hbox to3em{\hrulefill}\thinspace}
\providecommand{\MR}{\relax\ifhmode\unskip\space\fi MR }
\providecommand{\MRhref}[2]{%
  \href{http://www.ams.org/mathscinet-getitem?mr=#1}{#2}
} \providecommand{\href}[2]{#2}

\end{document}